\documentclass[12pt]{amsart}
\usepackage{amsmath,amssymb,amsfonts,a4wide}
\usepackage{epsfig}

\def\be{\begin{equation}}
\def\ee{\end{equation}}
\def\lbl{\label}

\newtheorem {thm}{Theorem}

\newtheorem {pr}[thm]{Proposition}

\def\E{{\mathbb E}}
\def\P{{\mathbb P}}
\def\eps{\epsilon}

\begin{document}

\title[Extremes  of sequences defined by random difference equations]{Convergence to type I distribution of the extremes of
sequences defined by random difference equation}
\author[Pawe{\l} Hitczenko]{Pawe{\l} Hitczenko$^\dag{}$}
\address{Pawe{\l} Hitczenko\\
Departments of Mathematics and Computer Science  \\
Drexel University\\
Philadelphia, PA 19104 \\
U.S.A}
\email{phitczenko@math.drexel.edu}
\urladdr{
http://www.math.drexel.edu/$\sim$phitczen}

\thanks{$^\dag{}$ Supported in part by the NSA grant  \#H98230-09-1-0062}

\date{\today}

\begin{abstract}
We study the extremes of a sequence of random variables  $(R_n)$ defined by the recurrence  $R_n=M_nR_{n-1}+q$, $n\ge1$, where
$R_0$ is arbitrary, $(M_n)$ are  iid copies of a
non--degenerate random variable $M$, $0\le M\le1$,  and $q>0$  is a
constant. We show  that under mild and natural conditions on $M$ the
suitably normalized extremes of $(R_n)$ converge in distribution to a double
exponential random variable. This partially complements a result of
de~Haan, Resnick, Rootz\'en, and de~Vries who considered  extremes of
the sequence $(R_n)$ under the assumption that $\P(M>1)>0$.\end{abstract} 

\maketitle

\section{Introduction}
We consider a special case of  the following random difference equation
 \be\lbl{perp-nq} R_n=Q_n+M_nR_{n-1},\quad n\ge1\ee where
$R_0$ is arbitrary and $(Q_n,M_n)$, $n\ge1$, are  i.i.d. copies of a two dimensional random vector 
$(Q,M)$, and $(Q_n,M_n)$ is independent of $R_{n-1}$. Later on we
specialize our discussion to a non--degenerate $M$, and $Q\equiv q$, a
positive constant. 
Much of the impetus for  studying  equations  like (\ref{perp-nq}) stems from numerous
applications of such schemes in mathematics and other
disciplines of science. 
We refer   to   \cite{eg,vervaat}  for  examples of fields in which equation (\ref{perp-nq}) have been of interest. Further examples of more recent applications are mentioned in  \cite{hw}, and  for examples of statistical issues arising in studying solutions of \eqref{perp-nq} see \cite{unc}. 

A fundamental theoretical result that goes back to Kesten  \cite{kesten} asserts that if 
\be\lbl{convcond} E\ln|M|<0\quad \mbox{and}\quad E\ln|Q|<\infty\ee
  then the sequence $(R_n)$ converges in distribution to a random variable $R$, which necessarily satisfies distributional identity
\be\lbl{id}R\stackrel d= MR+Q\ee
(see  also  \cite{vervaat} for a detailed discussion of the convergence properties of $(R_n)$).
In the same paper Kesten showed that if
$P(|M|>1)>0$ and \eqref{convcond} holds then, under some mild
additional conditions on $M$ and $Q$   the limiting
distribution is always heavy--tailed, that is, 
$\P(|R|>t)\sim
Ct^{-\kappa}$ for a suitably chosen $\kappa>0$.  
A different proof of this results was
given by Goldie in \cite{goldie}. By contrast,  it was shown in
\cite{gg} that in the complementary case $|M|\le1$  if  $|Q|\le q$
then the tail of $R$ has no slower than exponential decay.

Interestingly, much more work has been done in the heavy--tailed
situation. This is perhaps at least partially caused by the fact that
many of the processes appearing in applications (for example GARCH
processes in financial mathematics) are in fact
heavy--tailed. Nonetheless, the case $|M|\le1$ and $Q\equiv q$ contains a number of
interesting situations, including the class of Vervaat perpetuities, see
e.g. 
\cite{vervaat}. Vervaat perpetuities correspond to $M$ being a
$\operatorname{Beta}(\alpha,1)$ random variable for some $\alpha>0$ and $Q=1$ in which case one
gets
\be\lbl{verperp}R\stackrel d=1+M_1+M_1M_2+M_1M_2M_3+\dots\ee 
(some authors prefer not
to have a 1 at the beginning which corresponds to taking
$Q=M$). Particular cases of Vervaat perpetuities include the Dickman
distribution appearing in number theory (see \cite{dickman}), in the
analysis of  the limiting distribution of \texttt{Quickselect}
algorithm (see \cite{mms}), and in the limit theory of functionals of success
epochs in iid sequences of random variables
\cite[Section~4.7]{verber}.   Further connections are referenced
in \cite{ht} and we refer there for more
information.  For recent work on perfect
simulation of Vervaat perpetuities see \cite{fh} or \cite{df}.

In this note we will be interested in the extremal behavior of the sequence $(R_n)$. For any sequence of random variables $(Y_n)$ we let $(Y_n^*)$ be the sequence of partial maxima, i.e. $Y_n^*=\max_{k\le n}Y_k$, $n\ge1$.  With this notation, we will seek constants $(a_n)$ and $(b_n)$ such that  for all $x$
\be\lbl{maxconv} P(a_n(R_n^*-b_n)\le x)\longrightarrow G(x),\quad n\to\infty,\ee
where $G$ is a non--degenerate distribution function.

Under the  assumptions that $\P(M>1)>0$ the extremes of the sequence
$(R_n)$ when both $M$ and $Q$ are non--negative were studied  in
\cite{dhrrdv} and were shown to converge (after suitable
normalization) to Fr\'echet (i.e. Type II) distribution with parameter
$\kappa$. Here, we consider the complementary case, namely of a
light-tailed limiting distribution $R$.  Of course, in this situation
one expects convergence in \eqref{maxconv} to a Gumbel (i.e. a 
double--exponential or Type I) distribution, provided that there is
convergence at all.  The latter need not be the case, however. Indeed, if
$Q=1$ and $M$ has a two--point distribution $\P(M=1)=p=1-\P(M=0)$ then
as is seen from \eqref{verperp} $R$ has a geometric distribution with
parameter $1-p$ and thus no constants $(a_n)$, $(b_n)$ exist for which
 \eqref{maxconv} holds for a non--degenerate distribution $G$ (see \cite[Example~1.7.15]{llr}). Our
 main aim here is to show that under fairly general and natural
 conditions on $M$ (and for a degenerate $Q$) \eqref{maxconv} does
 hold for  suitable constants $(a_n)$, $(b_n)$ and a double
 exponential distribution $G(x)=\exp(-e^{-x})$, $-\infty<x<\infty$.  

\section{Extremal behavior}

Following the authors of
\cite{dhrrdv} we assume that both $M$ and $Q$ are non--negative. 
As  we mentioned earlier we  assume that $Q=q>0$ is a constant. 
 So,  
we consider
\be\lbl{perp-q}R_n=M_nR_{n-1}+q,\quad n\ge1,\quad\mbox{$R_0$ -- given},\ee 
where $M_n$ and $R_{n-1}$ on the right--hand side are independent and
where $(M_n)$ is a sequence of iid copies of a random variable $M$
satisfying
\be\lbl{Mass}0\le M\le 1,\quad  \mbox{$M$ -- non-degenerate}.\ee 
(The non--degeneracy assumption is to eliminate the possibility that $R$ itself
is degenerate.) Clearly, this  is more than \eqref{convcond} and thus implies
the convergence in
distribution of $(R_n)$. Furthermore, it has been known since
\cite{gg} that in that case the tail of the limiting  variable $R$ is
no heavier than exponential. Note that if $M$ is bounded away from 1
then $R$ is actually a bounded random variable. To exclude this
situation we assume that the right endpoint of $M$ is 1, that is
that  
\be\lbl{upend}\sup\{x:\ \P(M>x)>0\}=1.\ee  
Finally, we need to eliminate the possibility that $R$ is a geometric
variable. To this end it is enough 
to assume that 
\be\lbl{nongeom}\P(M=0)=0,\ee
since this guarantees that the distribution of $R$ is continuous (see e.g. \cite[Theorem~1.3]{air}).

We will prove the following theorem
\begin{thm}\lbl{mt}
Let $(R_n)$ satisfy \eqref{perp-q} with $M$ satisfying
\eqref{Mass}--\eqref{nongeom}. Then there exist sequences $(a_n)$
$(b_n)$ such that for every real $x$
\[\P(a_n(R_n^*-b_n)\le x)\to \exp(-e^{-x}),\quad\mbox{as  }
n\to\infty.\]
\end{thm}
\section{proof of Theorem~\ref{mt}}
We first outline our proof which generally follows the approach  of \cite{dhrrdv} (see also references therein for earlier developments). 
Writing out \eqref{perp-q} explicitly  we see that
\be\lbl{expl}
R_n=q+qM_n+qM_nM_{n-1}+\dots+qM_n\dots M_2+M_n\dots M_1R_0.\ee
Under our assumption \eqref{Mass} (as a matter of fact, under the
first part of \eqref{convcond} as well) the product $\prod_{k=1}^nM_k$
goes to 0 a.s.  Consequently, the extremal behavior of $(R_n)$ is the
same  regardless of the choice of the initial variable $R_0$. It is
particularly convenient to  choose $R_0$ so that it satisfies
\eqref{id} as then so does every $R_k$, $k\ge1$,  making the sequence
$(R_n)$ stationary.  Extremal behavior of stationary sequences is
 quite well understood (see e.g. \cite[Chapter~3]{llr}) and we will
 take advantage of  that. To find the extremal behavior of $(R_n)$ one has to do three things: 
\begin{itemize}
\item[(i)] analyse the extremal behavior of the associated
 independent sequence $(\hat{R}_n)$ consisting of iid random
 variables  equidistributed with $R$,  
\item[(ii)] verify that the
 sequence $(R_n)$ satisfies the $D(u_n)$ condition for 
  sequences $(u_n)$ of the form $u_n=b_n+x/a_n$, for any $x$ and
  suitably chosen sequences $(a_n)$, $(b_n)$, and 
\item[(iii)]  show that the sequence $(R_n)$ has the extremal index
  and find its value. 
\end{itemize}

Some of the difficulties with carrying out this program are caused by
the fact that,
contrary to the heavy -- tailed situation, much less is known about
the tail asymptotics in the  case of light tails. A notable exception
are Vervaat perpetuities (see \cite[Section~]{verber} for a
discussion). General results on light--tail case are scarce (see
\cite{gg,hw,h}) and less precise than Kesten's result in the heavy --
tailed situation. As a consequence, less precise information about
the norming constants $(a_n)$ $(b_n)$ will be available.   Our
substitute for Kesten's result  will be two--sided bounds obtained recently
in \cite{h}. 

We will treat  the three items above in separate subsections.

\subsection{Associated independent sequence}

We  appeal to the general theory of extremes as described in
e.g. \cite[Chapter
~1]{llr}. First,  we know from \cite[Theorem~1.3]{air} that
\eqref{nongeom} and non--degeneracy assumption on $M$ imply that $R$ has 
continuous distribution function $F_R$. Therefore, the condition (1.7.3) of Theorem~1.7.13 of \cite{llr} is
satisfied  and thus, for every 
$x>0$ there exist $u_n=u_n(x)$ such that
\be\lbl{limun}\lim_{n\to\infty}n\P(R>u_n)=e^{-x}.\ee    
In fact, since $R$ is continuous 
$u_n$ may be taken to be 
\[u_n(x)=F_{R}^{-1}(1-\frac{e^{-x}}n),\]
where $F_R $ is the probability distribution function of $R$. The
question now is whether $u_n$'s may be chosen to be linear functions
of $x$ i.e. whether there exist constants $a_n$ and $b_n$, $n\ge1$
such that for $x>0$ we have 
\be\lbl{linu}u_n(x)=\frac x{a_n}+b_n,\qquad n\ge1.\ee  
To address that question we will utilize a recent result of \cite{h}
which states that there exist absolute constants $c_i$, $i=0,1,2,3$ such that
for sufficiently large $y>0$
\[\exp\{c_0y\ln p_{\frac{c_1}y}\}\le\P(R>y)\le\exp\{c_2y\ln p_{\frac{c_3}y}\},
\]
where, following \cite{gg}, for $0<\delta<1$ we set
\be\lbl{pdelta}p_\delta=\P(1-\delta<M\le1)=1-F_M(1-\delta)\quad\mbox{and}\quad p_0=\lim_{\delta\to0}p_\delta=\P(M=1).\ee
Notice that by \eqref{upend} $p_\delta$ is strictly positive for $\delta\in(0,1)$.  
Now, if 
\[\P(R>u_n)=\frac{e^{-x}}n,\]
then 
\[\exp\{c_0u_n\ln p_{\frac{c_1}{u_n}}\}\le\frac{e^{-x}}n.\]
Therefore, if $w_n$ are chosen so that 
\[\exp\{c_0w_n\ln p_{\frac{c_1}{w_n}}\}=\frac{e^{-x}}n,\]
then $u_n\ge w_n$. By the same argument, if $v_n$ are such that 
\[\exp\{c_2v_n\ln p_{\frac{c_3}{v_n}}\}=\frac{e^{-x}}n,\]
then $\P(R>v_n)\le\frac{e^{-x}}n$ so that $u_n\le v_n$. 
Hence for every $x>0$
\[w_n(x)\le u_n(x)\le v_n(x)\]
and thus for every
$n\ge1$ there would exist  $0\le\alpha_n\le1$ such that 
\[u_n=\alpha_nw_n+(1-\alpha_n)v_n.\] 
If both $(v_n)$ and $(w_n)$ were linear, say, 
\[w_n(x)=\frac x{a_n'}+b_n',\quad v_n(x)=\frac
x{a_n''}+b_n'',\]
for some $(a_n')$, $(b_n')$, $(a_n'')$, and $(b_n'')$ then 
\eqref{linu}  would hold  with 
\[a_n=\left(\frac{\alpha_n}{a_n'}+\frac{1-\alpha_n}{a_n''}\right)^{-1}\quad 
\mbox{and} \quad b_n=\alpha_nb_n'+(1-\alpha_n)b_n''.\]
It therefore suffices to show the existence of linear norming for
partial maxima of iid random variables $(W_n)$ whose common
distribution $F_W$ satisfies
\[1-F_W(y)=\exp\{c_0y\ln p_{c_1/y}\},\quad \mbox{for}\quad y\ge y_0,\]
where $p_{c_1/y}$ is given by \eqref{pdelta} for some 
 fixed random variable $M$ satisfying \eqref{Mass}--\eqref{nongeom}.

In accordance with \cite[Theorem~1.5.1]{llr} to show that
\[\P(a_n'(W_n-b_n')\le x)\to\exp(-e^{-x}),\]
holds for every real $x$, the constants $(a_n')$ and $b_n')$ must be
constructed so that for every such $x$
\[n(1-F_W(b_n'+x/a_n'))\to e^{-x},\quad\mbox{as}\quad n\to\infty,\]
i.e. that 
\be\lbl{tailW}n\exp\left\{c_0(b_n'+\frac x{a_n'})\ln
  p_{\frac{c_1}{b_n'+x/a_n'}}\right\}\to e^{-x},\quad\mbox{as}\quad
n\to\infty.\ee
Choose $b_n'$ so that 
\be\lbl{bn}c_0b_n'\ln p_{c_1/b_n'}=-\ln n.\ee
 Then the
left--hand side of \eqref{tailW} is
\[\exp\left\{c_0b_n'\left(\ln
  p_{\frac{c_1}{b_n'+x/a_n'}}-\ln p_{c_1/b_n'}\right)\left(1+\frac x{a_n'b_n'}\right)+c_0\frac x{a_n'}\ln
  p_{c_1/b_n'}\right\}.\]
To choose $(a_n')$ first note that  the difference of logarithms
in the first summand 
is negative. Hence, if for any $n$, $a_n'\le-K\ln p_{c_1/b_n'}$ for some $K<c_0$
then the exponent is no more than $-xc_0/K<-x$. Therefore,
for any admissible choice of $(a_n')$ we must have $\liminf_na_n'/\ln
p_{c_1/b_n'}\le -c_0$ which implies in particular that
$a_n'b_n'\to\infty$. Thus, the exponent in the above formula is asymptotic to
\[c_0b_n'(\ln
  p_{\frac{c_1}{b_n'+x/a_n'}}-\ln p_{c_1/b_n'})+c_0\frac x{a_n'}\ln
  p_{c_1/b_n'}.\]
We can further assume that for each $n$ $1-c/b_n'$ is a
differentiability point of $F_M$ and that the derivative, $f_M$, is
finite at
$1-c_1/b_n'$. It then follows that the exponent is asymptotic
to 
\[-c_0\frac{c_1x}{a_n'b_n'p_{c_1/b_n'}}f_M(1-\frac{c_1}{b_n'})+c_0\frac x{a_n'}\ln
  p_{c_1/b_n'}
\]
and thus we may choose 
\be\lbl{an}a_n'=c_0\left(\frac{c_1}{b_n'p_{c_1/b_n'}}f_M(1-\frac{c_1}{b_n'})-\ln p_{c_1/b_n'}\right) .\ee

\subsection{$D(u_n)$ condition}

To check that $D(u_n)$ condition holds for sequences of the form $b_n+x/a_n$ we proceed in the same fashion as
\cite[proof of Theorem~2.1]{dhrrdv}; the argument there was, in turn,
based on \cite[proof of Lemma~3.1]{r-ap86}. Recall that, according to
\cite[Lemma~3.2.1(ii)]{llr} it suffices to
show that if
$1\le i_1<\dots<i_r<j_1<\dots< j_s\le n$ are such that $j_1-i_r\ge
\lambda n$ for $\lambda>0$ then
\[\P(\bigcap_{k=1}^r\{R_{i_k}\le u_n\}\cap\bigcap_{m=1}^s\{R_{j_m}\le
u_n\})-\P(\bigcap_{k=1}^r\{R_{i_k}\le
u_n\})\P(\bigcap_{k=1}^r\{R_{i_k}\le u_n\})\to0,\]
as $n\to\infty$.  
Set 
$I=\{i_1,\dots,i_r\}$ and $J=\{j_1,\dots,j_s\}$ and for any set $A$ of positive
integers let $R^*_A=\max_{a\in A}R_a$. 

It follows from \eqref{perp-q} that
 for $j>i$ we have 
\begin{eqnarray*}R_j&=&q+qM_j+\dots+qM_j\dots M_{i+2}+M_j\dots M_{i+1}R_i\\
&=:&S_{j,i}+M_j\cdot\dots\cdot M_{i+1}R_i,
\end{eqnarray*}
where, for $j>i$ we have
set 
\[S_{j,i}:=q+qM_j+\dots+qM_j\dots M_{i+2}.\]
 Hence, for  any $\eps_n>0$ 
we obtain 
\begin{eqnarray*}\{R_J^*\le u_n\}&=&\bigcap_{j\in J}\{S_{j,i_r}+M_j\dots M_{i_r+1}R_{i_r}\le u_n\}
\\&\supset& \bigcap_{j\in J}\{S_{j,i_r}\le u_n-\eps_n\}\cap\{M_j\dots M_{i_r+1}R_{i_r}\le \eps_n\}\\&
=&\bigcap_{j\in J}\{S_{j,i_r}\le u_n-\eps_n\}\setminus\bigcup_{j\in J}\{M_j\dots M_{i_r+1}R_{i_r}>\eps_n\}.
\end{eqnarray*}
Note that $R_k$ and $S_{n,m}$ are independent whenever,  $m\ge k$ so that $\{R_i:\ i\in I\}$  and $\{S_{j,i_r}:\ j\in J\}$ are independent, and hence we get
\[P(R_I^*\le u_n,R_J^*\le u_n)\ge P(R_I^*\le u_n)P(S_{J,i_r}^*\le u_n-\eps_n)-P(\bigcup_{j\in J}M_j\dots M_{i_r+1}R_{i_r}>\eps_n).
\]
Also, 
\[\{S^*_{J,i_r}\le u_n-\eps_n\}\supset\{R_J^*\le u_n-2\eps_n\}\cap\bigcap_{j\in J}\{M_j\dots M_{i_r+1}R_{i_r}\le\eps_n\},
\]
which further leads to 
\[P(R_I^*\le u_n,R_J^*\le u_n)\ge P(R_I^*\le u_n)P(R_J^*\le u_n-2\eps_n)-2P(\bigcup_{j\in J}M_j\dots M_{i_r+1}R_{i_r}>\eps_n).
\]
By essentially the same argument we also get 
\[P(R_I^*\le u_n,R^*_J\le u_n)\le P(R_I^*\le u_n)P(R_J^*\le u_n+2\eps_n)+2P(\bigcup_{j\in J}M_j\dots M_{i_r+1}R_{i_r}>\eps_n).
\] 
Combining these two estimates we obtain 
\begin{eqnarray*}&&|P(R_I^*\le u_n,R^*_J\le u_n)- P(R_I^*\le u_n)P(R_J^*\le u_n)|\\&&\quad\le
\max\{P(R_J^*\le u_n)-P(R_J^*\le u_n-2\eps_n), P(R_J^*\le u_n+2\eps_n)-P(R_J^*\le u_n)\}\\&&\qquad+2P(\bigcup_{j\in J}M_j\dots M_{i_r+1}R_{i_r}>\eps_n).
\end{eqnarray*}
Thus condition $D(u_n)$  will be verified once we show that both terms in the sum on the right--hand side
vanish as $n\to\infty$. 
To handle the first term we use stationarity and the fact that $j_s\le n$ to find that the maximum above is bounded by
\[\sum_{j\in J}P(u_n-2\eps_n\le R_j\le u_n+2\eps_n)\le nP(u_n-2\eps_n\le R\le u_n+2\eps_n).\]
Recall that $(u_n)$ satisfy \eqref{limun} and \eqref{linu}.
Thus, setting $\eps_n=\eps/a_n$ with $\eps>0$ sufficiently small we get
\[nP(u_n-2\eps_n\le R\le u_n+2\eps_n)\to
e^{-(x-2\eps)}-e^{-(x+2\eps)}=O(\eps).\]

Turning attention to  the second term,  using  $M_k\le1$ we see that
\begin{eqnarray*}\P(\bigcup_{j\in J}M_j\dots M_{i_r+1}R_{i_r}>\eps/a_n)&\le& \sum_{j\in J}\P(M_{j}\dots M_{i_r+1}R_{i_r}>\eps/a_n)\\&\le&
nP(M_{j_1-i_r}\dots M_1R_0>\eps/a_n).\end{eqnarray*}
Intersect the event underneath this
probability with $\{R>2b_n\}$ and its complement to see that this term  is bounded by
 \be\lbl{r0tail}n\P(R>2b_n)+n\P(M_{j_1-i_r}\dots
 M_1>\eps/(2a_nb_n)).\ee
Furthermore, since for any $T>0$ and  sufficiently large $n$,
$2b_n=b_n+\frac{a_nb_n}{a_n}>b_n+T/a_n$, the first term \eqref{r0tail} is bounded by
\[n\P(R>2b_n)\le n\P(R>b_n+T/a_n)\to e^{-T},
\] 
and thus goes to 0 upon letting $T\to\infty$. Turning to the second term
in \eqref{r0tail} we see that by Markov's inequality and independence
of $M_k$'s it is bounded by
\be\lbl{EMbound}\frac{2na_nb_n}\eps (EM)^{j_1-i_r}.\ee
We need to see that this vanishes as $n\to\infty$. 
 Recall that $EM<1$
and $j_1-i_r\ge\lambda n$ where $\lambda>0$, 
so that $(\E M)^{j_1-i_r}$ decays exponentially fast in $n$.
Furthermore, 
\[a_nb_n\le K\max\{a_n',a_n''\}\cdot\max\{b_n',b_n''\}\]
Recall that $b_n'$ and $b_n''$ satisfy \eqref{bn} (with different
constants). Thus they both are
$O(\ln n)$ as are $\ln p_{c_1/b_n'}$ and $\ln p_{c_3/b_n''}$. 
Hence,
\[a_n'\le
K\left(\frac{c_1}{b_n'}f_M(1-\frac{c_1}{b_n'})\frac1{p_{c_1/b_n'}}+\ln
  n\right).\]
Since $f_M$ is an integrable function, we may assume that
$\frac{c_1}{b_n'}f_M(1-c_1/b_n')=O(1)$ as $n\to\infty$. Finally, recall that $(b_n')$
satisfies \eqref{bn}. 
Therefore,
\[p_{c_1/b_n'}=\exp(-\frac{\ln n}{c_0b_n'})=n^{-1/c_0b_n'}\ge n^{-\alpha},\quad \alpha>0,\]
where the last inequality follows from the fact that $b_n'\to\infty$
as $n\to\infty$ which is evident from \eqref{bn}. 
It follows that $na_nb_n$ has a polynomial growth in $n$ and thus that
for every $\eps>0$ \eqref{EMbound} goes to 0 as $n\to\infty$.

\subsection{Extremal index}
We establish the following fact about  the extremal index of
$(R_n)$. It implies, in particular, that if $M$ does not have an atom
at 1, then  the extremal behavior of $(R_n)$ is exactly the same as it
would be for independent
$R_n$'s.

\begin{pr}
Let $(R_n)$ be a stationary sequence satisfying the recurrence
\eqref{perp-q}. Then $(R_n)$ has the extremal index $\theta$ whose
value is 
\[\theta=\limsup_n\P(MR+q\le u_n| R>u_n)=1-p_0=1-\P(M=1).\]
\end{pr}
\begin{proof}
Again following the authors of \cite{dhrrdv} we rely on Theorem~4.1 of
Rootz\'en \cite{r-jap88}. Since we have shown that $D(u_n)$ holds for
every sequence $u_n$ of the form $b_n+x/a_n$, $x>0$, it
remains to verify condition (4.3) of that theorem i.e. to show that
\[\limsup_{n\to\infty}|\P(R_{\lceil n\eps\rceil}\le
u_n|R_0>u_n)-\theta|\to0,\quad\mbox{as}\quad \eps\searrow0.\]
To  this end,  for given $\eps>0$, let $m:=m_\eps:=\lceil n\eps\rceil$. Then
 \[\P(R_m^*\le u_n|R_0> u_n)=\P(R_m\le u_n|R_{m-1}^*\le u_n,R_0>u_n)\P(R_{m-1}^*\le u_n|R_0>u_n).
 \]
By Markov property, for $m\ge2$ the first  probability on the right--hand side is  
\[\P(R_m\le u_n|R_{m-1}\le u_n)=\P(M_mR_{m-1}+q\le u_n|R_{m-1}\le u_n)=\P(MR+q\le u_n|R\le u_n).\]
Continuing in the same fashion we find that 
\begin{eqnarray*}\P(R_m^*\le u_n|R_0> u_n)&=&\P^{m-1}(MR+q\le u_n|R\le u_n)\P(R_1\le u_n|R_0> u_n)\\
&=&\left(1-\P(MR+q>u_n|R\le u_n)\right)^{m-1}\P(MR+q\le u_n|R> u_n).
\end{eqnarray*}
So, clearly
\[\limsup_n \P(R_m^*\le u_n|R_0> u_n)\le \limsup_n\P(MR+q\le u_n|R>
u_n).\]

On the other hand,
\begin{eqnarray*}n\P(MR+q>u_n|R\le u_n)&=&n\frac{\P(MR+q>u_n,R\le u_n)}{\P(R\le u_n)}\le
n\frac{\P(MR+q>u_n)}{1-\P(R>u_n)}\\
&=&n\frac{\P(R>u_n)}{1-\P(R>u_n)} \to e^{-x},\end{eqnarray*}
as $n\to\infty$ by the very choice of $(u_n)$. Thus 
\[\limsup_nn\P(MR+q>u_n|R\le u_n)\le e^{-x}=:c<\infty\]
so that
\[\liminf_n\left(1-\P(MR+q>u_n|R\le u_n)\right)^{m-1}\ge e^{-c\eps}\]
and hence
\[\limsup_n\P(R_m^*\le u_n|R_0\le u_n)\ge
e^{-c\eps}\limsup_n\P(MR+q\le u_n|R> u_n).\]
It follows that 
\begin{eqnarray*}&&\lim_{\eps\searrow0}\limsup_{n\to\infty}\left\{\left(1-\P(MR+q>u_n|R\le u_n)\right)^{m-1}\P(MR+q\le u_n|R> u_n)\right\}\\
&&\qquad\qquad
=\limsup_{n\to\infty}\P(MR+q\le u_n|R> u_n).\end{eqnarray*}

We now turn to evaluating 
\[\limsup_{n\to\infty}\P(MR+q\le u_n|R> u_n).
\] 
It is clear that if $p_0>0$ then for every $n$ such that $u_n\ge q$ we have $\P(MR+q\le u_n|R>u_n)=1-p_0$ so assume that $p_0=0$ and write 
\[\P(MR+q\le u_n|R> u_n)=1-\P(MR+q>u_n|R>u_n)=1-\frac{\P(MR+q>u_n,\
  R>u_n)}{\P(R>u_n)}.\]
It remains to show that the numerator in the last expression is of
lower order than the denominator.  To do that, let $(t_n)$ be a sequence converging to infinity but in such a way
that $t_n=o(b_n)$. Then 
\begin{eqnarray*}\P(MR+q>u_n,R>u_n)&=&\int_{u_n}^\infty \P(Mt+q>u_n)dF_R(t)\\&=&\left(\int_{u_n}^{u_n+t_n}+\int_{u_n+t_n}^\infty\right)\P(Mt+q>u_n)dF_R(t).
\end{eqnarray*}
Note that the probability underneath the integral is an increasing
function of $t$.  Bounding it trivially by 1 in the second term
we see that this term 
is
bounded by $\P(R>u_n+t_n)$.  This can be further bounded by
\[\P(R>b_n+\frac x{a_n}+t_n)=\P(R>b_n+\frac{x+a_nt_n}{a_n})\le\P(R>b_n+\frac{x+T}{a_n}),\]
whenever $a_nt_n\ge T$. It follows by the choice of $(u_n)$ and  $D(u_n)$ condition that 
\[\frac{\P(R>u_n+t_n)}{\P(R>u_n)}\le e^{-T},\]
for arbitrarily large $T$ and sufficiently large $n$ and thus it vanishes as $n\to\infty$.
The first integral is bounded by 
\[\P(M(u_n+t_n)+q>u_n)\P(u_n<R<u_n+t_n)\le
\P(M>1-\frac{t_n+q}{u_n+t_n})\P(R>u_n).\]
Since the first term goes to $p_0=0$ as $n\to\infty$,
 we see that 
 this term is  $o(\P(R>u_n))$
as $n\to\infty$. 
This shows that the extremal index is 1 when $p_0=0$ and completes the proof.
\end{proof}

\section{Remarks}
{\bf1.} The main drawback of Theorem~1 is that it does not give a good handle on the norming
constants $(a_n)$ and $(b_n)$. This is generally caused by a lack of precise
information about the tails of the limiting random variable
$R$. However, even in the rare cases in which more precise
information about tails of $R$ is available, the formulas seem to be too
complicated to  make the precise statements about 
$(a_n)$ and $(b_n)$ practical. For example, when $q=1$ and $M$ has
$\operatorname{Beta}(\alpha,1)$ distribution, $\alpha>0$, (i.e. $R$ is a Vervaat
perpetuity) then Vervaat \cite[Theorem~4.7.7]{verber} (based on earlier arguments of de Bruijn
\cite{dB}) found the expression for the density of $R$. This, in principle, could be used to get precise enough asymptotics of the tail function of $R$ and thus determine the asymptotic values of $(b_n)$ and $(a_n)$. However, the nature of these formulas, makes obtaining explicit asymptotic expressions for $(a_n)$ and $(b_n)$ difficult if not impossible.  As far as we know, Vervaat perpetuities provide the only class of examples (within our restrictions on  $M$ and $Q$) for which the asymptotics of the tail function is known.  On the other hand, Theorem~\ref{mt}  typically gives the order of the magnitude of $(a_n)$ and $(b_n)$.

{\bf2.} The expression \eqref{an}  for $(a_n')$  often
simplifies to $a_n'\sim -c_0\ln p_{c_1/b_n'}$ (with corresponding simplification for $(a_n)$). This will happen, for
example, whenever $p_0=0$ and $\delta f_M(1-\delta)/p_\delta$ is bounded
as $\delta\to0$, in particular, when $M$ is
$\operatorname{Beta}(\alpha,\beta)$ random variable, $\alpha,\beta>0$. In that
case, $b_n'$ may be chosen to be asymptotic to $\frac{\ln
  n}{c_0\beta\ln\ln n}$  and then
$a_n'\sim c_0\beta\ln\ln n$.  Hence, $(a_n)$ and $(b_n)$ are of order
$\ln\ln n$ and $\ln n/\ln\ln n$, respectively. Note that Vervaat perpetuity corresponds to $\beta=1$ and Dickman distribution to $\alpha=\beta=1$.

{\bf3.} There are, however, situations for which the above remark is not true. The following situation was considered in  \cite[Theorem~6]{hw}.
Let $M$ have density given by
\[f_M(t)=K\exp\{-\frac1{(1-t^r)^{1/(r-1)}}\},\quad 1<r<\infty,\quad 0<t<1,\]
where $K=K_r$ is a normalizing constant. Then, as $\delta\searrow0$,
\[p_\delta\sim(1-(1-\delta)^r)^{r/(r-1)}\exp\{-(1-(1-\delta)^r)^{-1/(r-1)}\}\sim(r\delta)^{r/(r-1)}\exp\{-(r\delta)^{-1/(r-1)}\},\]   
so that 
\[\frac{c_1f_M(1-c_1/b_n')}{b_n'p_{c_1/b_n'}}\sim\left(\frac{b_n'}{c_1r^r}\right)^{1/(r-1)}.\]
On the other hand, 
\[-\ln p_{c_1/b_n'}\sim\left(\frac{b_n'}{c_1}\right)^{1/(r-1)}+\frac r{r-1}\ln(b_n'/rc_1)=
\left(\frac{b_n'}{c_1}\right)^{1/(r-1)}
\left(1+O(\frac{\ln b_n'}{b_n'^{1/(r-1)}})\right),
\]
so that both terms appearing in \eqref{an} are of the same order. Here
again, the norming constants $(a_n)$, $(b_n)$ in Theorem~\ref{mt} may
be determined up to absolute multiplicative factors and are of order
$(\ln n)^{1/r}$ and $(\ln n)^{(r-1)/r}$, respectively. 

{\bf4.} Consider another example from \cite{hw} in which 
\[f_M(t)=K\exp\left(-\int_{1-t}^1\frac{e^{1/s}}sds\right),\quad 0<t<1.\]
Then (see \cite[Lemma~8]{hw}) $\ln p_\delta\sim -\delta e^{1/\delta}$
as $\delta\to0$. Similarly, one can check that 
\[\frac{\delta f_M(1-\delta)}{p_\delta}\sim\frac{\delta e^{-\delta
    e^{1/\delta}}e^{\delta e^{1/\delta}}}{\delta e^{-1/\delta}}=
e^{1/\delta},\]
so this time the first term in the expression \eqref{an} is of 
higher order than the second. 
 It follows from the asymptotics above that $a_n'\sim(\ln n)/c_0c_1$ and
$b_n'\sim c_1\ln\ln n$ and  hence $(a_n)$, $(b_n)$ are of order $\ln n$
and $\ln\ln n$, respectively.  

\bibliographystyle{plain}
\bibliography{maxperp}

\begin{thebibliography}{10}

\bibitem{air}
G.~Alsmeyer, A.~Iksanov, and U.~R{\"o}sler.
\newblock On distributional properties of perpetuities.
\newblock {\em J. Theoret. Probab.}, 20:666--682, 2009.

\bibitem{unc}
C.~Baek, V.~Pipiras, H.~Wendt, and P.~Abry.
\newblock Second order properties of distribution tails and estimation of tail
  exponents in random difference equations.
\newblock {\em Extremes}, 12:361--400, 2009.

\bibitem{dB}
N.~G. de~Bruijn.
\newblock The asymptotic behaviour of a function occurring in the theory of
  primes.
\newblock {\em Archiv f\"or Matematik Astronomi och Fysik}, 22:1--14, 1930.

\bibitem{dhrrdv}
L.~de~Haan, S.~I. Resnick, H.~Rootz{\'e}n, and C.~G. de~Vries.
\newblock Extremal behaviour of solutions to a stochastic difference equation
  with applications to {ARCH} processes.
\newblock {\em Stochastic Process. Appl.}, 32(2):213--224, 1989.

\bibitem{df}
L.~Devroye and O.~Fawzi.
\newblock Simulating the {D}ickman distribution.
\newblock {\em Statist. Probab. Lett.}, 80:242--247, 2010.

\bibitem{dickman}
K.~Dickman.
\newblock On the frequency of numbers containing prime factors of a certain
  relative magnitude.
\newblock {\em J. Indian Math. Soc.}, 15:25--32, 1951.

\bibitem{eg}
P.~Embrechts and C.~M. Goldie.
\newblock Perpetuities and random equations.
\newblock In {\em Asymptotic statistics (Prague, 1993)}, Contrib. Statist.,
  pages 75--86. Physica, Heidelberg, 1994.

\bibitem{fh}
J.~A. Fill and M.~L. Huber.
\newblock Perfect simulation of {V}ervaat perpetuities.
\newblock {\em Electron. J. Probab.}, 15:96--109, 2010.

\bibitem{goldie}
C.~M. Goldie.
\newblock Implicit renewal theory and tails of solutions of random equations.
\newblock {\em Ann. Appl. Probab.}, 1(1):126--166, 1991.

\bibitem{gg}
C.~M. Goldie and R.~Gr{\"u}bel.
\newblock Perpetuities with thin tails.
\newblock {\em Adv. in Appl. Probab.}, 28:463--480, 1996.

\bibitem{h}
P.~Hitczenko.
\newblock On tails of perpetuities.
\newblock {\em J. Appl. Probab.}, 47:1191 -- 1194, 2010.

\bibitem{hw}
P.~Hitczenko and J.~Weso{\l}owski.
\newblock Perpetuities with thin tails, revisited.
\newblock {\em Ann. Appl. Probab.}, 19:2080 -- 2101, 2009.
\newblock Corrected version available at http://arxiv.org/abs/0912.1694.

\bibitem{ht}
H.-K. Hwang and T.-H. Tsai.
\newblock Quickselect and the {D}ickman function.
\newblock {\em Combin. Probab. Comput.}, 11:353--371, 2002.

\bibitem{kesten}
H.~Kesten.
\newblock Random difference equations and renewal theory for products of random
  matrices.
\newblock {\em Acta Math.}, 131:207--248, 1973.

\bibitem{llr}
M.~R. Leadbetter, G.~Lindgren, and H.~Rootz{\'e}n.
\newblock {\em Extremes and related properties of random sequences and
  processes}.
\newblock Springer Series in Statistics. Springer-Verlag, New York, 1983.

\bibitem{mms}
H.~M. Mahmoud, R.~Modarres, and R.~T. Smythe.
\newblock Analysis of {QUICKSELECT}: an algorithm for order statistics.
\newblock {\em RAIRO Inform. Th\'eor. Appl.}, 29:255--276, 1995.

\bibitem{r-ap86}
H.~Rootz{\'e}n.
\newblock Extreme value theory of moving average processes.
\newblock {\em Ann. Probab.}, 14:612--652, 1986.

\bibitem{r-jap88}
H.~Rootz{\'e}n.
\newblock Maxima and exceedances of stationary {M}arkov chains.
\newblock {\em J. Appl. Probab.}, 20:371--390, 1988.

\bibitem{verber}
W.~Vervaat.
\newblock {\em Success epochs in {B}ernoulli trials (with applications in
  number theory)}.
\newblock Mathematisch Centrum, Amsterdam, 1972.
\newblock Mathematical Centre Tracts, No. 42.

\bibitem{vervaat}
W.~Vervaat.
\newblock On a stochastic difference equation and a representation of
  nonnegative infinitely divisible random variables.
\newblock {\em Adv. in Appl. Probab.}, 11(4):750--783, 1979.

\end{thebibliography}

\end{document}